\pdfoutput=1
\documentclass[final,leqno,onefignum,onetabnum,onealgnum,oneeqnum,onethmnum]{siamart171218}
\usepackage{lipsum}
\usepackage{amsfonts}
\usepackage{amsmath}
\usepackage{amssymb}
\usepackage{bm}
\usepackage{graphicx}
\usepackage{xcolor}
\usepackage{epstopdf}
\usepackage[caption=false]{subfig}
\graphicspath{{imgs/}}
\usepackage[algo2e,ruled,vlined,resetcount,linesnumbered,algosection]{algorithm2e}
\crefname{algocf}{alg.}{algs.}
\Crefname{algocf}{Algorithm}{Algorithms}
\Crefname{Algorithm}{Algorithm}{Algorithms}
\crefformat{figure}{#2{}\scshape{Fig.} #1{}#3}
\crefformat{lemma}{#2{}\scshape{Lemma} #1{}#3}
\crefformat{theorem}{#2{}\scshape{Theorem} #1{}#3}
\crefformat{proposition}{#2{}\scshape{Proposition} #1{}#3}
\crefformat{definition}{#2{}\scshape{Definition} #1{}#3}

\Crefformat{lemma}{#2{}\scshape{Lemma} #1{}#3}
\Crefformat{theorem}{#2{}\scshape{Theorem} #1{}#3}
\Crefformat{proposition}{#2{}\scshape{Proposition} #1{}#3}
\Crefformat{definition}{#2{}\scshape{Definition} #1{}#3}

\crefformat{remark}{#2{}\scshape{Remark} #1{}#3}
\usepackage{algpseudocode}
\AppendGraphicsExtensions{.tif}
\usepackage{transparent}
\usepackage{tikz}
\usetikzlibrary{positioning,calc}
\definecolor{selfDefinedColor}{rgb}{0.8,0.1,0.4}
\usepackage{mwe}
\tikzset{blankboximg/.style={remember picture,white,ultra thick,draw,inner sep=0pt,outer sep=0pt}}
\tikzset{redboximg/.style={remember picture,red,ultra thick,draw,inner sep=0pt,outer sep=0pt}}
\tikzset{blueboximg/.style={remember picture,blue,ultra thick,draw,inner sep=0pt,outer sep=0pt}}
\tikzset{selfDefinedColorboximg/.style={remember picture,selfDefinedColor,ultra thick,draw,inner sep=0pt,outer sep=0pt,dashed}}
\tikzset{greenboximg/.style={remember picture,green,ultra thick,draw,inner sep=0pt,outer sep=0pt,dashed}}
\usepackage{etoolbox}
\apptocmd{\sloppy}{\hbadness 10000\relax}{}{}
\title{Alternating minimization for a single step TV-Stokes model for image denoising}
\author{
    Bin Wu\thanks{Department of Computer Science, Electrical Engineering and Mathematical Sciences, Western Norway University of Applied Sciences, Inndalsveien 28, 5063 Bergen, Norway (\email{bin.wu@hvl.no}, \email{talal.rahman@hvl.no}).}
    \and
    Xue-Cheng Tai\thanks{Department of Mathematics, University of Bergen, All\'{e}gaten 41, 5007 Bergen, Norway (\email{xue-cheng.tai@uib.no}).}
    \and
    Talal Rahman\footnotemark[1]
}
\headers{Alternating minimization for single step TV-Stokes}{B. Wu, X. C. Tai, and T. Rahman}

\begin{document}
\maketitle

\begin{abstract}
  The paper presents a fully coupled TV-Stokes model, and propose an algorithm based on an alternating minimization of the objective functional, whose first iteration is exactly the same as the modified TV-Stokes model proposed earlier. The model is also a generalization of the second order Total Generalized Variation model. A convergence analysis is given.
\end{abstract}



\section{Introduction}
The ROF model or the Rudin-Osher-Fatemi model, cf. \cite{Rudin1992}, is one of the most well-known and successful models in variational image processing. It handles various problems, as for instance, denoising, inpainting, based on a total variation minimization. The model, however, suffers from a so-called staircase effect, which results in block artifacts. Among the models that alleviate the effect, are the TV-Stokes, cf. \cite{Rahman2007,Elo2009a}, and its modified variant, cf. \cite{Litvinov2011}, henceforth called the LRT model. The modified TV-Stokes model consists of two steps where the object image is reconstructed based on the vector field obtained from its first step. The vector field is achieved by a second-order minimization subject to a divergence-free condition, also known as the Stokes constraint, which implies that the vector field is a tangent field. For simplicity, we consider the gradient field and the condition becomes a curl-free condition, in other words condition for a potential function.  

The modified TV-Stokes is defined as follows. Given an observed image $f$, define its gradient as $\mathbf{n}$, and an orthogonal projector $\Pi$ as $\Pi(\mathbf{n}) = \nabla \triangle^{\dagger} \nabla \cdot (\mathbf{n})$, as in \cite{Elo2009a}. The two-step modified TV-Stokes model thus has an equivalent form, cf. \cite{Elo2009a,Rahman2007}, as follows. 
\begin{enumerate}
  \item[1.] Find the smoothed gradient field with curl-free constraint:
        \begin{eqnarray*}
          \min_{\substack{\mathbf{n} \\ \Pi \mathbf{n} = \mathbf{n}}} \vert \nabla \mathbf{n} \vert_{(\Omega)} + \frac{1}{2\delta} \Vert \mathbf{n} - \nabla f \Vert_{(\Omega)}^2.
        \end{eqnarray*}
  \item[2.] Restore the object image using the gradient field obtained from the first step:
        \begin{eqnarray*}
          \min_u \vert \nabla u - \mathbf{n} \vert_{(\Omega)} + \frac{1}{2\theta} \Vert u - f \Vert_{(\Omega)}^2,
        \end{eqnarray*}
\end{enumerate}
where $\Omega$ is a bounded open domain in $\mathbb{R}^2$. $\vert \cdot \vert_{(\Omega)}$ and $\Vert \cdot \Vert_{(\Omega)}$ are the usual L$^1$ and L$^2$ norms, respectively.
The current paper is about combining the two steps of the modified TV-Stokes into one step, coupling the two variables, that is the image and its normal field in one minimization step. It is then easier to develop an iterative regularization algorithm for the combined model. 

We organize the paper as follows: we propose the combined model and an iterative algorithm for the model in \cref{sec:prop}, give a convergence analysis in \cref{sec:con}, and finally end the paper with some remarks at the end.

\section{Proposed Model and Iterative Algorithm}\label{sec:prop}
We propose our one-step model as follows.
\begin{eqnarray}\label{eqs:proposed}
  \min_{\substack{\mathbf{n}, u \\ \Pi \mathbf{n} = \mathbf{n}}} \alpha \vert \nabla \mathbf{n} \vert_{(\Omega)} + \beta \vert\nabla u - \mathbf{n}\vert_{(\Omega)} + \frac{\eta_1}{2} \Vert \mathbf{n} - \nabla f \Vert_{(\Omega)}^2  + \frac{\eta_2}{2} \Vert u - f \Vert_{(\Omega)}^2.
\end{eqnarray}
However, we solve the model by iterations. We split the minimization problem of the model into two sub-problems by freezing one of the two variables which we then solve alternatingly in each iteration. The sub-problems are described in the following sub section.
\subsection{Sub-problems}
In the section we introduce the two sub-problems, which we solve alternatingly.
\begin{enumerate}
  \item[\bf{1:}] Find the smoothed gradient field satisfying the potential function constraint, by solving the following minimization problem:
        \begin{eqnarray}\label{eqs:sub1}
          \min_{\substack{\mathbf{n} \\ \Pi \mathbf{n} = \mathbf{n}}} \alpha \vert \nabla \mathbf{n} \vert_{(\Omega)} + \frac{\eta_1}{2} \Vert \mathbf{n} - \nabla f \Vert_{(\Omega)}^2 + \beta \vert \mathbf{n} - \nabla u \vert_{(\Omega)}.
        \end{eqnarray}
  Here $u$ is fixed, which is the solution of the second sub-problem. Initially $u$ is set equal to $f$.
  \item[\bf{2:}] Restore the given image by solving the following minimization problem:
        \begin{eqnarray}\label{eqs:sub2}
          \min_u \beta \vert\nabla u - \mathbf{n}\vert_{(\Omega)} + \frac{\eta_2}{2} \Vert u - f \Vert_{(\Omega)}^2.
        \end{eqnarray}
        Here $\mathbf{n}$ is fixed, and comes from the first step in the current iteration.
\end{enumerate}

For the first sub-problem, we turn the constrained minimization into an unconstrained minimization. In the following \Cref{lemma:equivalency1}, we prove the equivalency between the two formulations.
\begin{lemma}[Equivalency]\label{lemma:equivalency1}
  For all $\mathbf{n} \in \mathbb{R}^2$, define $\Pi$ as a projector such that $\Pi(\mathbf{n}) = \nabla \triangle^{\dagger} \nabla \cdot \mathbf{n}$. The constrained problem (\ref{eqs:sub1}) is equivalent to the following unconstrained problem
  \begin{eqnarray}\label{eqs:sub1_equivalent}
    \min_{\mathbf{n}} \alpha \vert \nabla \Pi \mathbf{n} \vert_{(\Omega)} + \frac{\eta_1}{2} \Vert \mathbf{n} - \nabla f \Vert_{(\Omega)}^2 + \beta \vert \Pi \mathbf{n} - \nabla u \vert_{(\Omega)}.
  \end{eqnarray}
\end{lemma}
\begin{proof}
  Let $\mathbf{p}$ and $\mathbf{q}$ be the dual variables such that $\mathbf{p} \in C_c^1(\Omega, \mathbb{R}^4)$ and $\mathbf{q} \in C_c^1(\Omega, \mathbb{R}^2)$, with $\Vert \mathbf{p} \Vert_{\infty}$ and $\Vert \mathbf{q} \Vert_{\infty} \leq 1$, the minimization problem (\ref{eqs:sub1}) is thus
  \begin{eqnarray}\label{eqs:sub1_reform}
    \min_{\mathbf{n}} \max_{\substack{\mathbf{p}, \mathbf{q} \\ \vert \cdot \vert \leq 1}} \alpha \langle \mathbf{n}, \nabla \cdot \mathbf{p} \rangle + \frac{\eta_1}{2} (\mathbf{n} - \nabla f)^2 + \beta\langle \mathbf{n} - \nabla u, \mathbf{q} \rangle + \lambda \cdot (\Pi \mathbf{n} - \mathbf{n}),
  \end{eqnarray}
  where $\lambda$ is the Lagrange multiplier. By the min-max theorem, cf. \cite{Sion1957}, we can reformulate the min-max into max-min, above. The corresponding Euler-Lagrange equations become
  \begin{eqnarray}\label{eqs:el_sub_1}
    \alpha \nabla \cdot \mathbf{p} + \eta_1 (\mathbf{n} - \nabla f) + \beta \mathbf{q} + (\Pi - I) \lambda = 0,
  \end{eqnarray} and
  \begin{eqnarray*}
    \Pi \mathbf{n} - \mathbf{n} = 0.
  \end{eqnarray*}
  Applying the $\Pi$ projection on both sides of (\ref{eqs:el_sub_1}), we obtain the following relation.
  \begin{eqnarray*}
    \eta_1 (\mathbf{n} - \nabla f) = - \alpha \Pi \nabla \cdot \mathbf{p} - \beta \Pi \mathbf{q}.
  \end{eqnarray*} Substituting back in (\ref{eqs:el_sub_1}), we obtain
  \begin{eqnarray*}
    (\Pi - I) \lambda =  - \alpha (I - \Pi) \nabla \cdot \mathbf{p} - \beta (I - \Pi) \mathbf{q}.
  \end{eqnarray*}
  The problem (\ref{eqs:sub1_reform}) is accordingly
  \begin{eqnarray*}
    \min_{\mathbf{n}}\max_{\substack{\mathbf{p}, \mathbf{q} \\ \vert \cdot \vert \leq 1}} \alpha \langle \mathbf{n}, \Pi \nabla \cdot \mathbf{p}\rangle + \frac{\eta_1}{2} (\mathbf{n} - \nabla f)^2 + \beta\langle \mathbf{n}, \Pi \mathbf{q} \rangle - \beta\langle \nabla u, \mathbf{q} \rangle,
  \end{eqnarray*} which is equivalent to the following primal problem:
  \begin{eqnarray*}
    \min_{\mathbf{n}} \alpha \vert\nabla \Pi \mathbf{n}\vert_{(\Omega)} + \frac{\eta_1}{2} \Vert \mathbf{n} - \nabla f \Vert_{(\Omega)}^2 + \beta \vert\Pi \mathbf{n} - \nabla u\vert_{(\Omega)}.
  \end{eqnarray*}\qed
\end{proof}

\subsection{Algorithm}
We solve the optimization problem by the alternating minimization method, cf. \cite{Ortega1987,Bertsekas1997}, where the method is also known as block-nonlinear Gauss–Seidel method or the block coordinate descent method, cf. \cite{Bertsekas1999} and cf. \cite{Beck2015} for a review. The algorithm is given as follows (\cref{alg:proposed}).

\begin{algorithm}[H]
  \caption{Alternating Minimization TV-Stokes}\label{alg:proposed}
  \SetKw{Initialize}{initialize}
    \Initialize{$k = 0$, $u_0 = f$\;}\\
    \Repeat{satisfied}{
    $k = k + 1$\;\\
    $\mathbf{n}_k = \mbox{arg} \min_{\mathbf{n}} \{ \alpha \vert\nabla \Pi \mathbf{n}\vert(\Omega) + \frac{\eta_1}{2} \Vert \mathbf{n} - \nabla f \Vert_{(\Omega)}^2 + \beta \vert\Pi \mathbf{n} - \nabla u_{k-1}\vert(\Omega) \}$\;\\
    $u_k = \mbox{arg} \min_{u} \{ \beta \vert\nabla u - \mathbf{n}_k\vert(\Omega) + \frac{\eta_2}{2} \Vert u-f\Vert_{(\Omega)}^2 \} $\;
    }
    \Return{$u$.}
\end{algorithm}

\subsection{Solving Sub-problem 1 for \texorpdfstring{$\mathbf{n}$}{n}}
We adopt Chambolle's semi-implicit algorithm \cite{Chambolle1997} to solve the proposed sub-prob-\allowbreak lems. The dual problem for the sub-problem (\ref{eqs:sub1_equivalent}) is as follows: 
\begin{eqnarray}\label{eqs:sub1_dual}
  \min_{\mathbf{n}}\max_{\substack{\mathbf{p}, \mathbf{q} \\ \vert \cdot \vert \leq 1}} \alpha \langle \mathbf{n}, \Pi \nabla \cdot \mathbf{p}\rangle + \frac{\eta_1}{2} (\mathbf{n} - \nabla f)^2 + \beta\langle \mathbf{n}, \Pi \mathbf{q} \rangle  - \beta\langle \nabla u, \mathbf{q} \rangle.
\end{eqnarray}
We first consider the minimization by the min-max theorem, cf. \cite{Sion1957}, which gives us the optimal condition on $\mathbf{n}$ as follows.
\begin{eqnarray}\label{eqs:sub1_min}
  \mathbf{n} = \nabla f - \frac{\alpha}{\eta_1} \Pi \nabla \cdot \mathbf{p} - \frac{\beta}{\eta_1} \Pi \mathbf{q}.
\end{eqnarray}
Accordingly the maximization problem is
\begin{eqnarray}\label{eqs:sub1_max}
  \max_{\substack{\mathbf{p}, \mathbf{q} \\ \vert \cdot \vert \leq 1}} \langle \nabla f, \alpha \Pi \nabla \cdot \mathbf{p} + \beta \Pi \mathbf{q} \rangle - \frac{1}{2\eta_1} (\alpha \Pi \nabla \cdot \mathbf{p} + \beta \Pi \mathbf{q})^2 - \beta\langle \nabla u, \mathbf{q} \rangle.
\end{eqnarray}
The line searching directions $\mathbf{r}$ for $\mathbf{p}$ and $\mathbf{q}$ in Chambolle's semi-implicit algorithm are respectively
\begin{eqnarray}\label{eqs:sub1_gp}
  \mathbf{r_p} = \nabla (\frac{\alpha}{\eta_1} \Pi \nabla \cdot \mathbf{p} + \frac{\beta}{\eta_1} \Pi \mathbf{q} - \nabla f)
\end{eqnarray} and
\begin{eqnarray}\label{eqs:sub1_gq}
  \mathbf{r_q} = \nabla f - \nabla u - \frac{\alpha}{\eta_1} \Pi \nabla \cdot \mathbf{p} - \frac{\beta}{\eta_1} \Pi \mathbf{q}.
\end{eqnarray}

\subsection{Solving Sub-problem 2 for \texorpdfstring{$u$}{u}}
We use the same approach as for the sub-problem 1, but this time for the scalar function $u$. Accordingly, the dual problem of the sub-problem (\ref{eqs:sub2}) is as follows:
\begin{eqnarray}\label{eqs:sub2_dual}
  \min_u \max_{\mathbf{s}} \beta \langle \nabla u - \mathbf{n}, - \mathbf{s} \rangle + \frac{\eta_2}{2} (u - f)^2,
\end{eqnarray} where $\mathbf{s} \in C_c^1(\Omega, \mathbb{R}^2)$ and $\Vert \cdot \Vert_{\infty} \leq 1$.
The minimization problem results in optimal $u$ as follows:
\begin{eqnarray}\label{eqs:sub2_min}
  u = f - \frac{\beta}{\eta_2} \nabla \cdot \mathbf{s},
\end{eqnarray} The maximization problem becomes
\begin{eqnarray}\label{eqs:sub2_max}
  \max_{\mathbf{s}}
  \beta \langle f, \nabla \cdot \mathbf{s} \rangle + \beta \langle \mathbf{n}, \mathbf{s} \rangle - \frac{\beta^2}{2\eta_2} (\nabla \cdot \mathbf{s})^2.
\end{eqnarray}
The line search direction $\mathbf{r}$ for $\mathbf{s}$ is
\begin{eqnarray}\label{eqs:sub2_gq}
  \mathbf{r_s} = \nabla(\frac{\beta}{\eta_2} \nabla \cdot \mathbf{s} - f) + \mathbf{n}.
\end{eqnarray}

\section{Convergence Analysis}\label{sec:con}
We prove the convergence following the approach in the paper \cite{Beck2015}.

\subsection{Definitions and Assumptions}
Define the vector $\mathbf{x} = (\mathbf{n},u)$ such that $\mathbf{n} \in \mathbb{R}^2(\Omega) \cap \{\mathbf{n}:\Pi \mathbf{n} = \mathbf{n}\}$ and $u \in \mathbb{R}(\Omega)$. Rewrite the problem (\ref{eqs:proposed}) as:
\begin{eqnarray}\label{eqs:g_and_l}
  \min_{\mathbf{x}} \{ H(\mathbf{x}) := g(\mathbf{x}) + l(\mathbf{x}) \},
\end{eqnarray} where $g(\mathbf{x}) := g(\mathbf{n}, u) = \alpha \vert \nabla \Pi \mathbf{n} \vert_{(\Omega)} + \beta \vert\nabla u - \Pi \mathbf{n} \vert_{(\Omega)}$ and $l(\mathbf{x}) := l(\mathbf{n}, u) = \frac{\eta_1}{2} \Vert \mathbf{n} - \nabla f \Vert_{(\Omega)}^2  + \frac{\eta_2}{2} \Vert u - f \Vert_{(\Omega)}^2$.

Define functions $g_1$, $g_2$, $l_1$ and $l_2$ respectively as
\begin{align*}
  g_1(\mathbf{n}) & := \alpha \vert \nabla \Pi \mathbf{n} \vert_{(\Omega)} + \beta \vert\Pi  \mathbf{n} - \nabla u\vert_{(\Omega)}, \\
  g_2(u) & := \beta \vert\nabla u - \Pi \mathbf{n}\vert_{(\Omega)}, \\
  l_1(\mathbf{n}) &  := \frac{\eta_1}{2} \Vert \mathbf{n} - \nabla f \Vert_{(\Omega)}^2, \\
  l_2(u) & := \frac{\eta_2}{2} \Vert u - f\Vert_{(\Omega)}^2.
\end{align*}
We assume that $g_1$, $g_2$, $l_1$ and $l_2$ satisfy the following properties:
\begin{itemize}
  \item [(a)]
        Functions $g_1:\mathbb{R}^2(\Omega)\rightarrow(-\infty, \infty)$ and $g_2:\mathbb{R}(\Omega)\rightarrow(-\infty, \infty)$ are closed and proper convex and subdifferentiable.
  \item [(b)]
        Functions $l_1$ and $l_2$ are continuously differentiable convex functions over domains $\mbox{dom}$$(g_1)$ and $\mbox{dom}$$(g_2)$ respectively. The partial derivatives of $l$ with respect to $\mathbf{n}$ and $u$ are denoted as $\nabla_1 l(\mathbf{x})$ and $\nabla_2 l(\mathbf{x})$ respectively.
  \item [(c)]
        The gradient of $l$ is Lipschitz continuous with respect to $\mathbf{n}$ over domain $\mbox{dom}$$(g_1)$ with a constant $L_1 \in (0, \infty)$ and with respect to $u$ over domain $\mbox{dom}$$(g_2)$ with a constant $L_2 \in (0, \infty)$ as follows
        \begin{eqnarray*}
          \Vert \nabla_1 l(\mathbf{n} + \mathbf{d}_1, u) - \nabla_1 l(\mathbf{n}, u) \Vert_{(\Omega)} \leq L_1 \Vert \mathbf{d}_1 \Vert_{(\Omega)},
        \end{eqnarray*} and
        \begin{eqnarray*}
          \Vert \nabla_2 l(\mathbf{n}, u + d_2) - \nabla_2 l(\mathbf{n}, u) \Vert_{(\Omega)} \leq L_2 \Vert d_2\Vert_{(\Omega)},
        \end{eqnarray*} where $\mathbf{d}_1 \in \mathbb{R}^2(\Omega)$ and $u \in \mathbb{R}(\Omega)$ such that $\mathbf{n} + \mathbf{d}_1 \in \mbox{dom}$$(g_1)$ and $u + d_2 \in \mbox{dom}$$(g_2)$.
  \item [(d)]
        The optimal solution set of problem (\ref{eqs:g_and_l}), denoted by $X^*$, is nonempty, and the corresponding energy value of $H(\mathbf{x})$ is $H^*$. In addition, the solution sets for sub-problems are also nonempty.
\end{itemize}
Using the notations above, we reformulate \Cref{alg:proposed} as follows.

\begin{algorithm}[H]
  \caption{Alternating Minimization Method (Reformulated)} \label{alg:proposed_2}
  \SetKwBlock{Initialize}{initialize}{}
    \Initialize{
    $k = 0$\;\\
    $\mathbf{n}_0 \in \mbox{dom}$ $g_1$, $u_0 = f \in \mbox{dom}$ $g_2$, such that $u_0 \in \mbox{argmin}_{u \in BV(\Omega)}$ $ g_2(\mathbf{n}_0,u) + l_2(\mathbf{n}_0,u)$\;}
    \Repeat{satisfied}{
    $k = k + 1$\;\\
    $\mathbf{n}_k = \mbox{arg} \min_{\mathbf{n} \in BV^2(\Omega)} g_1(\mathbf{n},u_{k-1}) + l_1(\mathbf{n})$\;\\
    $u_k = \mbox{arg} \min_{u \in BV(\Omega)} g_2(\mathbf{n}_k,u) + l_2(u)$\;}{}
    \Return{$u.$}
\end{algorithm}

The $k$-th iteration reads as $\mathbf{x}_k = (\mathbf{n}_k,u_k)$, and the $(k+1)$-th iteration as $\mathbf{x}_{k+1} = (\mathbf{n}_{k+1},u_{k+1})$. The intermediate iteration is defined as half iteration, i.e., $\mathbf{x}_{k+1/2} = (\mathbf{n}_{k+1},u_k)$. The sequence generated by the optimal problems has the relation:
\begin{eqnarray*}
  H(\mathbf{x}_0) \ge H(\mathbf{x}_{1/2}) \ge H(\mathbf{x}_1) \ge H(\mathbf{x}_{3/2}) \ge \ldots
\end{eqnarray*}

\subsection{Preliminaries}
Three concepts are associated with the convergence analysis of the proposed algorithm: the proximal mapping, the gradient mapping, and the block descent lemma.

Define proximal operator $\mbox{prox}_h(\cdot)$ as
\begin{eqnarray*}
  \mbox{prox}_h(\mathbf{z}) = \mbox{arg}\min_y h(\mathbf{y}) + \frac{1}{2}\Vert\mathbf{y} - \mathbf{z}\Vert^2,
\end{eqnarray*}
where $h:\mathbb{R}^n\rightarrow(-\infty,\infty]$ is a closed, proper convex function. A lemma on proximal mapping, cf. \cite{Beck2015}, is useful in the our proof of convergence.
\begin{lemma}\label{lemma:proximal}
  Let $h:\mathbb{R}^n\rightarrow(-\infty,\infty]$ be a closed, proper and convex function, $M > 0$, and $\mathbf{v} = \mbox{prox}_{\frac{1}{M}h}(\mathbf{z})$. Thus
  if $\mathbf{y} \in \mbox{dom}(h)$ then
  \begin{eqnarray*}
    h(\mathbf{y}) \ge h(\mathbf{v}) + M\langle \mathbf{z} - \mathbf{v}, \mathbf{y} - \mathbf{v} \rangle.
  \end{eqnarray*}
\end{lemma}

\begin{proof}
  The optimal condition for the convex functional $h(\mathbf{y}) + \frac{M}{2} \Vert\mathbf{y} - \mathbf{z} \Vert^2$ is $\partial h(\mathbf{v}) + M (\mathbf{v} - \mathbf{z}) \ni 0$, where $\partial h(\mathbf{v})$ is the subgradient at $\mathbf{v}$. By the definition of subgradient, it follows that
  \begin{eqnarray*}
    h(\mathbf{y}) \ge h(\mathbf{v}) + M\langle \mathbf{z} - \mathbf{v}, \mathbf{y} - \mathbf{v} \rangle,
  \end{eqnarray*}
  and hence the proof.\qed
\end{proof}

Now define prox-grad mapping $T_M(\cdot)$ as $T_M(\mathbf{y}) = \mbox{prox}_{\frac{1}{M}h}(\mathbf{y} - \frac{1}{M}\nabla f(\mathbf{y}))$ for $M > 0$ associated with $h(\mathbf{y}) + f(\mathbf{y}) $ where $f(\mathbf{y}) = \frac{M}{2} \Vert\mathbf{y} - \mathbf{z} \Vert^2$, cf. \cite{Nesterov2004}. The corresponding gradient mapping is defined as $G_M(\mathbf{y}) = M(\mathbf{y} - T_M(\mathbf{y})) = M(\mathbf{y} - \mbox{prox}_{\frac{1}{M}h}(\mathbf{y} - \frac{1}{M}\nabla f(\mathbf{y})))$.
Note that if $G_M(\mathbf{y}) = 0$ for some $M > 0$, $\mathbf{y}$ is an optimal solution. For the problem (\ref{eqs:g_and_l}), the gradient mapping is accordingly
\begin{eqnarray*}
  G_M(\mathbf{x}) = M(\mathbf{x} - T_M(\mathbf{x})) = M(\mathbf{x} - \mbox{prox}_{\frac{1}{M}g}(\mathbf{x} - \frac{1}{M}\nabla l(\mathbf{x}))),
\end{eqnarray*} where $M\in \{ \eta_1, \eta_2 \}$, is the parameter corresponding to $\mathbf{n}$ and $u$. The partial gradient mappings are therefore
\begin{eqnarray*}
  G_M^1(\mathbf{x}) & = \eta_1 (\mathbf{n} - T_{\eta_1}(\mathbf{n})) = \eta_1(\mathbf{n} - \mbox{prox}_{\frac{1}{\eta_1}g_1}(\mathbf{n} - \frac{1}{\eta_1}\nabla l_1(\mathbf{n}))), \\ G_M^2(\mathbf{x}) & = \eta_2 (u - T_{\eta_2}(u) = \eta_2(u - \mbox{prox}_{\frac{1}{\eta_1}g_2}(u - \frac{1}{\eta_2}\nabla l_2(u))),
\end{eqnarray*} where
\begin{eqnarray*}
  G_M(\mathbf{x}) = (G_M^1(\mathbf{x}),G_M^2(\mathbf{x})).
\end{eqnarray*}

\begin{lemma}[block descent lemma]\label{lemma:blockdescent}
  For all $M_1 > L_1$ and $M_2 > L_2$, and the assumption (c), the following relations hold.
  \begin{eqnarray*}
    l(\mathbf{n} + \mathbf{d_1}, u) & \le l(\mathbf{n}, u) + \langle \nabla_1 l(\mathbf{n}, u), \mathbf{d_1} \rangle + \frac{M_1}{2} \Vert \mathbf{d_1} \Vert^2, \\ l(\mathbf{n}, u + d_2) & \le l(\mathbf{n}, u) + \langle \nabla_2 l(\mathbf{n}, u), d_2 \rangle + \frac{M_2}{2} \Vert d_2 \Vert^2.
  \end{eqnarray*}
\end{lemma}
\begin{proof}
  By the Taylor expansion, we have
  \begin{eqnarray*}
    l(\mathbf{n} + \mathbf{d_1}, u) & = l(\mathbf{n}, u) + \langle \nabla_1 l(\mathbf{n}, u), \mathbf{d_1} \rangle + \mathbf{d_1}^{\top} \frac{ \nabla_1 \nabla_1 l(\tilde{\mathbf{n}}, u) }{2} \mathbf{d_1}, \\ l(\mathbf{n}, u + d_2) & = l(\mathbf{n}, u) + \langle \nabla_2 l(\mathbf{n}, u), d_2 \rangle + \frac{\nabla_2 \nabla_2 l(\mathbf{n}, \tilde{u})}{2} \Vert d_2 \Vert^2,
  \end{eqnarray*}
  where $\tilde{\mathbf{n}}$ is some vector in the box $[\mathbf{n},\mathbf{n} + \mathbf{d_1}]$, and $\tilde{u}$ is some scalar value in $[u, u + d_2]$. By the Lipschitz assumption, with $M_1 > L_1$ and $M_2 > L_2$, the following relations thus hold:
  \begin{eqnarray*}
    l(\mathbf{n} + \mathbf{d_1}, u) & \le l(\mathbf{n}, u) + \langle \nabla_1 l(\mathbf{n}, u), \mathbf{d_1} \rangle + \frac{M_1}{2} \Vert \mathbf{d_1} \Vert^2, \\ l(\mathbf{n}, u + d_2) & \le l(\mathbf{n}, u) + \langle \nabla_2 l(\mathbf{n}, u), d_2 \rangle + \frac{M_2}{2} \Vert d_2 \Vert^2.
  \end{eqnarray*}
  The proof follows.\qed
\end{proof}

\begin{lemma}[sufficient decrease, cf. \cite{Beck2010,Beck2015}]\label{lemma:sff_decrease}
 \begin{sloppypar}
  Let $s \in C^{L,1}(\Omega, \mathbb{R}^p \rightarrow \mathbb{R})$, $h: \mathbb{R}^p \rightarrow (-\infty, \infty]$ be a closed, proper, and convex sub-differentiable function, and $S(\cdot) = s(\cdot) + h(\cdot)$, then
  \end{sloppypar}
  \begin{eqnarray*}
    S(x) - S(\mbox{prox}_{\frac{1}{L}h} ( x - \frac{\nabla s(x)}{L} ) ) \ge \frac{1}{2L} \Vert L ( x - \mbox{prox}_{\frac{1}{L}h} ( x - \frac{\nabla s(x)}{L} ) ) \Vert^2. 
  \end{eqnarray*}
\end{lemma}

\begin{proof}
  Since $h$ is convex and sub-differentiable, we have
  \begin{eqnarray*}
      h(x) \ge h( \mbox{prox}_{\frac{h}{L}} ( x - \frac{\nabla s(x)}{L} ) ) + (x - \mbox{prox}_{\frac{h}{L}} ( x - \frac{\nabla s(x)}{L}  ) ) \partial h(\mbox{prox}_{\frac{h}{L}} ( x - \frac{\nabla s(x)}{L} ) ) ).
  \end{eqnarray*}
  The optimal condition for the minimization of $\mbox{prox}_{\frac{1}{L}h} ( x - \frac{1}{L} \nabla s(x) )$ gives
  \begin{eqnarray*}
    & \partial h( \mbox{prox}_{\frac{1}{L}h} ( x - \frac{1}{L}\nabla s(x) ) ) \ni L ( x - \frac{1}{L}\nabla s(x) - \mbox{prox}_{\frac{1}{L}h} ( x - \frac{1}{L}\nabla s(x) ) ).
  \end{eqnarray*}
  We therefore have the following inequality
  \begin{eqnarray*}
    h(x) \ge h( \mbox{prox}_{\frac{h}{L}} ( x - \frac{\nabla s(x)}{L} ) ) + \frac{1}{L} \Vert L ( x - \mbox{prox}_{\frac{h}{L}} ( x - \frac{\nabla s(x)}{L}  ) ) \Vert^2 &  \\ - (x - \mbox{prox}_{\frac{h}{L}} ( x - \frac{\nabla s(x)}{L} )) \nabla s(x). &
  \end{eqnarray*}
  Further, we obtain
  \begin{eqnarray*}
    S(x) - S( \mbox{prox}_{\frac{1}{L}h} ( x - \frac{1}{L}\nabla s(x) ) ) & \ge s(x)  - s( \mbox{prox}_{\frac{1}{L}h} ( x - \frac{1}{L}\nabla s(x) ) ) \\ + & \frac{1}{L} \Vert L ( x - \mbox{prox}_{\frac{1}{L}h} ( x - \frac{1}{L} \nabla s(x) ) ) \Vert^2 \\ - & (x - \mbox{prox}_{\frac{1}{L}h} ( x - \frac{1}{L}\nabla s(x) )) \nabla s(x).
  \end{eqnarray*}
  By the second order Taylor expansion and the Lipschitz continuous condition, cf. \Cref{lemma:blockdescent},
  \begin{eqnarray*}
    s( \mbox{prox}_{\frac{1}{L}h} ( x - \frac{1}{L}\nabla s(x) ) ) \le s(x) + ( \mbox{prox}_{\frac{1}{L}h} ( x - \frac{1}{L}\nabla s(x) ) - x ) \nabla s(x) \\ + \frac{1}{2L} \Vert L ( x - \mbox{prox}_{\frac{1}{L}h} ( x - \frac{1}{L} \nabla s(x) ) ) \Vert^2.
  \end{eqnarray*}
  This leads to
  \begin{eqnarray*}
    S(x) - S( \mbox{prox}_{\frac{1}{L}h} ( x - \frac{1}{L}\nabla s(x) ) ) \ge \frac{1}{2L} \Vert L( x - \mbox{prox}_{\frac{1}{L}h} ( x - \frac{1}{L} \nabla s(x) ) ) \Vert ^2.
  \end{eqnarray*}
  The proof thus follows.\qed
\end{proof}

Applying this to our problem, we can conclude that
\begin{eqnarray}\label{eqs:sff_decrease_n}
  H(\mathbf{n},u) - H( T_{\eta_1}(\mathbf{n}), u)
  \ge \frac{1}{2\eta_1}  \Vert G_{\eta_1}^1(\mathbf{n}, u) \Vert^2,
\end{eqnarray} and
\begin{eqnarray}\label{eqs:sff_decrease_u}
  H(\mathbf{n},u) - H( \mathbf{n}, T_{\eta_2}(u) )
  \ge \frac{1}{2\eta_2} \Vert G_{\eta_2}^2(\mathbf{n}, u) \Vert^2.
\end{eqnarray}

\subsection{Convergence Rate}
\begin{lemma}\label{lemma:half_contraction}
  Let $\{ \mathbf{x}_k \}$ be the sequence generated by \cref{alg:proposed_2}. Then $\forall k \ge 0$
  \begin{eqnarray*}
    H(\mathbf{x}_{k + \frac{1}{2}}) - H(\mathbf{x}^*) \le \Vert G_{\eta_1}^1(\mathbf{x}_k) \Vert \cdot \Vert \mathbf{x}_k - \mathbf{x}^* \Vert,
  \end{eqnarray*} and
  \begin{eqnarray*}
    H(\mathbf{x}_{k + 1}) - H(\mathbf{x}^*) \le \Vert G_{\eta_2}^2(\mathbf{x}_{k + \frac{1}{2}}) \Vert \cdot \Vert \mathbf{x}_{k + \frac{1}{2}} - \mathbf{x}^* \Vert.
  \end{eqnarray*}
\end{lemma}
\begin{proof}
  {\sloppy Since $\mathbf{x}_{k + \frac{1}{2}} = (\mathbf{n}_{k + 1}, u_k)$ minimizes the energy functional \newline  $H(\mathbf{n},u_k)$, we have
  \begin{eqnarray}\label{eqs:half_contraction}
    H(\mathbf{x}_{k + \frac{1}{2}}) - H(\mathbf{x}^*) \le H(T_{\eta_1}(\mathbf{x}_k)) - H(\mathbf{x}^*).
  \end{eqnarray}
  Meanwhile, we have $T_{\eta_1}(\mathbf{x}_k) = (T_{\eta_1}^1(\mathbf{x}_k),T_{\eta_1}^2(\mathbf{x}_k)) = (T_{\eta_1}^1(\mathbf{x}_k),u_k - G_{\eta_1}^2(\mathbf{x}_k))$ which is the same as $ T_{\eta_1}(\mathbf{x}_k) = (T_{\eta_1}^1(\mathbf{x}_k),u_k)$ by the optimal condition $G_M^2(\mathbf{x}_k) = 0$ for all $M > 0$. Inequality (\ref{eqs:half_contraction}) is thus rewritten as
  \begin{eqnarray*}
    H(\mathbf{x}_{k + \frac{1}{2}}) - H(\mathbf{x}^*) \le H(T_{\eta_1}^1(\mathbf{x}_k),u_k) - H(\mathbf{x}^*).
  \end{eqnarray*}}
  By \Cref{lemma:blockdescent}, we have that
  \begin{eqnarray*}
    l(T_{\eta_1}^1(\mathbf{x}_k),u_k) - l(\mathbf{x}^*) & \le l(\mathbf{x}_k) + \langle \nabla_1 l(\mathbf{x}_k), T_{\eta_1}^1(\mathbf{x}_k) - \mathbf{n}_k \rangle \\ & + \frac{\eta_1}{2} \Vert T_{\eta_1}^1(\mathbf{x}_k) - \mathbf{n}_k \Vert ^2 - l(\mathbf{x}^*) \\ & = l(\mathbf{x}_k) + \langle \nabla_1 l(\mathbf{x}_k), T_{\eta_1}(\mathbf{x}_k) - \mathbf{x}_k \rangle \\ & + \frac{\eta_1}{2} \Vert T_{\eta_1}^1(\mathbf{x}_k) - \mathbf{n}_k \Vert ^2 - l(\mathbf{x}^*),
  \end{eqnarray*} and, by the convexity $l(\mathbf{x}_k) - l(\mathbf{x}^*) \le \langle \nabla l(\mathbf{x}_k), \mathbf{x}_k - \mathbf{x}^* \rangle$, that
  \begin{eqnarray*}
    & l(T_{\eta_1}^1(\mathbf{x}_k)) - l(\mathbf{x}^*) \le \langle \nabla_1 l(\mathbf{x}_k), T_{\eta_1}(\mathbf{x}_k) - \mathbf{x}^* \rangle + \frac{\eta_1}{2} \Vert T_{\eta_1}^1(\mathbf{x}_k) - \mathbf{n}_k \Vert ^2.
  \end{eqnarray*}
  Meanwhile, for $h$, by the convexity and the optimal condition for $T_{\eta_1}(\mathbf{x}_k)$, that is $\partial_1 h(T_{\eta_1}(\mathbf{x}_k)) \ni \eta_1(\mathbf{x}_k - \frac{1}{\eta_1}\nabla_1 l(\mathbf{x}_k) - T_{\eta_1}(\mathbf{x}_k))$, we have
  \begin{eqnarray*}
    & h(\mathbf{x}^*) \ge h(T_{\eta_1}(\mathbf{x}_k)) +  \langle \partial_1 h(T_{\eta_1}(\mathbf{x}_k)), \mathbf{x}^* - T_{\eta_1}(\mathbf{x}_k) \rangle, \\ & h(\mathbf{x}^*) \ge h(T_{\eta_1}(\mathbf{x}_k)) +  \eta_1 \langle \mathbf{x}_k - \frac{1}{\eta_1}\nabla_1 l(\mathbf{x}_k) - T_{\eta_1}(\mathbf{x}_k), \mathbf{x}^* - T_{\eta_1}(\mathbf{x}_k) \rangle.
  \end{eqnarray*}
  From \cref{eqs:half_contraction}, we get
  \begin{align*}
    H(\mathbf{x}_{k + \frac{1}{2}}) - H(\mathbf{x}^*) \le & l(T_{\eta_1}^1(\mathbf{x}_k),u_k) - l(\mathbf{x}^*) + h(T_{\eta_1}(\mathbf{x}_k)) - h(\mathbf{x}^*) \\ \le & \eta_1 \langle \mathbf{x}_k - T_{\eta_1}(\mathbf{x}_k), T_{\eta_1}(\mathbf{x}_k) - \mathbf{x}^* \rangle + \frac{\eta_1}{2} \Vert T_{\eta_1}^1(\mathbf{x}_k) - \mathbf{n}_k \Vert ^2 \\ = & - \frac{1}{2 \eta_1}\Vert G_{\eta_1}(\mathbf{x}_k)\Vert ^2 + \langle G_{\eta_1}(\mathbf{x}_k), \mathbf{x}_k - \mathbf{x}^* \rangle \\ \le & \Vert G_{\eta_1}(\mathbf{x}_k)\Vert \cdot \Vert \mathbf{x}_k - \mathbf{x}^* \Vert \\ = & \Vert G_{\eta_1}^1(\mathbf{x}_k)\Vert \cdot \Vert \mathbf{x}_k - \mathbf{x}^* \Vert.
  \end{align*}
  By a similar analysis, we can also show that
  \begin{eqnarray*}
    H(\mathbf{x}_{k + 1}) - H(\mathbf{x}^*) \le \Vert G_{\eta_2}^2(\mathbf{x}_{k + \frac{1}{2}}) \Vert \cdot \Vert \mathbf{x}_{k + \frac{1}{2}} - \mathbf{x}^* \Vert,
  \end{eqnarray*}
  and the proof follows.\qed
\end{proof}

\begin{lemma}\label{lemma:full_relation1}
  Let $\{ \mathbf{x}_k \}$ be the sequence generated by \cref{alg:proposed_2}. Then $\forall k \ge 0$
  \begin{eqnarray*}
    H(\mathbf{x}_k) - H(\mathbf{x}_{k + 1}) \ge \frac{1}{ 2 \min\{ \eta_1, \eta_2\} M^2 }(H(\mathbf{x}_{k + 1}) - H^*) ^2,
  \end{eqnarray*}
  where $M := \max \Vert \mathbf{x} - \mathbf{x}^* \Vert$.
\end{lemma}
\begin{proof}
  By \Cref{lemma:sff_decrease} and \Cref{lemma:half_contraction}, we have
  \begin{align*}
    H(\mathbf{x}_k) - H(\mathbf{x}_{k + 1}) & \ge H(\mathbf{x}_k) - H(\mathbf{x}_{k + \frac{1}{2}}) \\ & \ge H(\mathbf{x}_k) - H(T_{\eta_1}(\mathbf{x}_{k})) \\ & \ge \frac{1}{2\eta_1} \Vert G_{\eta_1}^1(\mathbf{x}_k) \Vert^2 \\ & \ge \frac{(H(\mathbf{x}_{k + \frac{1}{2}}) - H(\mathbf{x}^*)) ^2}{2\eta_1 M^2} \\ & \ge \frac{(H(\mathbf{x}_{k + 1}) - H(\mathbf{x}^*)) ^2}{2\eta_1 M^2}.
  \end{align*}
  Similarly,
  \begin{align*}
    H(\mathbf{x}_k) - H(\mathbf{x}_{k + 1}) & \ge H(\mathbf{x}_{k + \frac{1}{2}}) - H(\mathbf{x}_{k + 1}) \\ & \ge H(\mathbf{x}_{k + \frac{1}{2}}) - H(T_{\eta_2}(\mathbf{x}_{k + \frac{1}{2}})) \\ & \ge \frac{1}{2\eta_2} \Vert G_{\eta_2}^2(\mathbf{x}_{k + \frac{1}{2}}) \Vert^2 \\ & \ge \frac{(H(\mathbf{x}_{k + 1}) - H(\mathbf{x}^*)) ^2}{2\eta_2 M^2}.
  \end{align*}
  Therefore,
  \begin{eqnarray*}
    H(\mathbf{x}_k) - H(\mathbf{x}_{k + 1}) \ge \frac{1}{ 2 \min\{ \eta_1, \eta_2\} M^2 }(H(\mathbf{x}_{k + 1}) - H^*) ^2.
  \end{eqnarray*}
  Hence, the lemma is proved.\qed
\end{proof}

\begin{lemma}\label{lemma:full_relation2}
  Let $\{ A_k\}_{k\ge0}$ be a nonnegative sequence of real numbers such that $A_k - A_{k+1} \ge \gamma A_{k+1}^2$, $A_1 \le \frac{1.5}{\gamma}$ and $A_2 \le \frac{1.5}{2 \gamma}$ for some positive $\gamma$. Then $A_k \le \frac{1.5}{\gamma} \frac{1}{k}$ for $k > 0$.
\end{lemma}
\begin{proof}
  We use induction to prove.
  For $k = 1,2$, it holds.
  Assume it also holds for $k$, that is $A_k \le \frac{1.5}{\gamma} \frac{1}{k}$.
  For case $k + 1$, we have $\gamma A_{k + 1}^2 + A_{k + 1} \le A_k \le \frac{1.5}{\gamma} \frac{1}{k}$, which is same as saying $(\gamma A_{k + 1} + \frac{1}{2})^2 \le \frac{1.5}{k} + \frac{1}{4}$.
  Therefore
  \begin{eqnarray*}
  \gamma A_{k + 1} \le \sqrt{\frac{1.5}{k} + \frac{1}{4}} - \frac{1}{2} = \frac{1.5}{\sqrt{0.5k - 1 + (\frac{k}{2} + 1)^2} + \frac{k}{2}}.
  \end{eqnarray*}
  Since $k \ge 2$, we have $\gamma A_{k + 1} \le \frac{1.5}{k + 1}$. The proof thus follows.\qed
\end{proof}

\begin{theorem}\label{theorem:full_relation2}
  Let $\{ \mathbf{x}_k \}_{k \ge 0}$ be the sequence generated by \cref{alg:proposed_2}. Then $\forall k \ge 1$
  \begin{eqnarray*}
    H(\mathbf{x}_k) - H^* \le \frac{ \max\{2(H(\mathbf{x}_0) - H^*), 3\min\{\eta_1, \eta_2 \}M^2\}}{k},
  \end{eqnarray*}
  where $M := \max \Vert \mathbf{x} - \mathbf{x}^* \Vert$.
\end{theorem}
\begin{proof}
  Set $A_k = H(\mathbf{x}_k) - H^*$ and $\tilde{\gamma} = 1/(2 \min\{ \eta_1, \eta_2\} M^2)$. By \Cref{lemma:full_relation1}, we have
  \begin{eqnarray*}
    A_k - A_{k + 1} \ge \tilde{\gamma} A_{k + 1}^2.
  \end{eqnarray*}
  Since $A_1 = H(\mathbf{x}_1) - H^* \le H(\mathbf{x}_0) - H^*$ and $A_2 = H(\mathbf{x}_2) - H^* \le H(\mathbf{x}_0) - H^*$, by setting $\gamma = \min\{ \frac{1.5}{2(H(\mathbf{x}_0) - H^*)}, \frac{1}{2 \min\{ \eta_1, \eta_2\} M^2} \}$, we have
  \begin{eqnarray*}
    A_k - A_{k + 1} \ge \gamma A_{k + 1}^2,
  \end{eqnarray*} and $A_1 \le \frac{1.5}{\gamma}$ and $A_2 \le \frac{1.5}{2\gamma}$. By \Cref{lemma:full_relation2}, we obtain
  \begin{eqnarray*}
    H(\mathbf{x}_k) - H^* \le \frac{\max\{2(H(\mathbf{x}_0) - H^*), 3 \min\{ \eta_1, \eta_2\} M^2 \}}{k},
  \end{eqnarray*}
  Hence the lemma is proved.\qed
\end{proof}

\begin{lemma}\label{lemma:full_relation3}
  Let $\{ A_k\}_{k\ge0}$ be a nonnegative sequence of real numbers such that $A_k - A_{k+1} \ge \gamma A_{k+1}^2$. Then $A_k \le \max \{ (\frac{1}{2})^{\frac{k - 1}{2}} A_0, \frac{4}{\gamma (k - 1)}  \}$ for all $k > 1$, and $\forall \varepsilon > 0$, $A_k \le \varepsilon$ if $k \ge \max \{ \frac{2}{\mbox{ln}(2)}(\mbox{ln}(A_0) - \mbox{ln}(\varepsilon)), \frac{4}{\gamma \varepsilon}\} + 1$
\end{lemma}
\begin{proof}
  By the fact that $A_k - A_{k+1} \ge \gamma A_{k+1}^2$, we have
  \begin{eqnarray*}
    \frac{1}{A_{k+1}} - \frac{1}{A_k} = \frac{A_k - A_{k+1}}{A_kA_{k+1}} \ge \gamma \frac{A_{k+1}}{A_k}.
  \end{eqnarray*}
  For the case $A_{k+1} > \frac{1}{2}A_k$, we have
  \begin{eqnarray*}
    \frac{1}{A_{k+1}} - \frac{1}{A_k} \ge \frac{\gamma}{2}.
  \end{eqnarray*}
  Suppose that $k$ is even and there are at least $\frac{k}{2}$ elements satisfy  $A_{k+1} > \frac{1}{2}A_k$, then $\frac{1}{A_k} \ge \frac{k \gamma}{4}$, i.e. $A_k \le \frac{4}{k \gamma}$. If there are less than $\frac{k}{2}$ elements satisfying $A_{k+1} > \frac{1}{2}A_k$, which implies that there are at least $\frac{k}{2}$ elements $A_{k+1} \le \frac{1}{2}A_k$, then $A_k \le (\frac{1}{2})^{\frac{k}{2}} A_0$. Therefore for both cases, we have $A_k \le \max\{ (\frac{1}{2})^{\frac{k}{2}} A_0, \frac{4}{k \gamma} \}$. If $k$ is odd, then $A_k \le A_{k - 1} \le \max\{ (\frac{1}{2})^{\frac{k-1}{2}} A_0, \frac{4}{(k-1) \gamma} \}$. By comparison, we obtain $A_k \le \max\{ (\frac{1}{2})^{\frac{k-1}{2}} A_0, \frac{4}{(k-1) \gamma} \}$ for all $k > 1$.

  In order to guarantee the inequality $A_k \le \varepsilon$, we must have
  \begin{eqnarray*}
    \max\{ (\frac{1}{2})^{\frac{k-1}{2}} A_0, \frac{4}{(k-1) \gamma} \} \le \varepsilon.
  \end{eqnarray*}
  It leads to a set of two inequalities $(\frac{1}{2})^{\frac{k-1}{2}} A_0 \le \varepsilon,$ and $\frac{4}{(k-1) \gamma}  \le \varepsilon,$ which is the same as
  \begin{eqnarray*}
    k \ge \frac{2}{\mbox{ln}(2)} (\mbox{ln}(A_0) - \mbox{ln}(\varepsilon)) + 1,
  \end{eqnarray*} and
  \begin{eqnarray*}
    k \ge \frac{4}{\varepsilon \gamma}  + 1.
  \end{eqnarray*}
  Therefore,
  \begin{eqnarray*}
    k \ge \max \{ \frac{2}{\mbox{ln}(2)} (\mbox{ln}(A_0) - \mbox{ln}(\varepsilon)), \frac{4}{\varepsilon \gamma} \} + 1
  \end{eqnarray*} is sufficient to guarantee $A_k \le \varepsilon$.\qed
\end{proof}

\begin{theorem}
  Let $\{ \mathbf{x}_k \}_{k \ge 0}$ be the sequence generated by \cref{alg:proposed_2}. Then $\forall k \ge 2$
  \begin{eqnarray*}
    H(\mathbf{x}_k) - H^* \le \max\{ (\frac{1}{2})^{\frac{k-1}{2}}(H(\mathbf{x}_0) - H^*), \frac{8 \min\{\eta_1, \eta_2 \}M^2}{k - 1} \},
  \end{eqnarray*}
  and $H(\mathbf{x}_k) - H^* \le \varepsilon$ is obtained after
  \begin{eqnarray*}
    \max\{\frac{2}{\mbox{ln}(2)} (\mbox{ln}(H(\mathbf{x}_0) - H^*) - \mbox{ln}(\varepsilon)), \frac{8 \min\{\eta_1, \eta_2 \}M^2}{\varepsilon}\} + 1,
  \end{eqnarray*} iterations
  where $M := \max \Vert \mathbf{x} - \mathbf{x}^* \Vert$.
\end{theorem}
\begin{proof}
  By \Cref{lemma:full_relation1}, and setting $A_k = H(\mathbf{x}_k) - H^*$ and $\gamma = \frac{1}{2\min\{ \eta_1, \eta_2 \}M^2}$, we have $A_k - A_{k+1} \ge \gamma A_{k + 1}^2$. The proof then follows directly from \Cref{lemma:full_relation3}.\qed
\end{proof}

\section{Conclusions}\label{sec:conc}
We proposed a model combining the two minimization steps of the original TV-Stokes for image denoising, into a single minimization step. The model optimizes a functional with respect to the intensity of the objective image and its gradient in a single step model. We have applied an alternating minimization algorithm to solve the model. The first iteration of the algorithm is exactly the modified variant of the TV-Stokes, cf. \cite{Litvinov2011}. The model is also equivalent to the second order Total Generalized Variation (TGV), cf. \cite{Knoll2011}, when $\eta_1 = 0$.

\bibliographystyle{siamplain}
\bibliography{refs}

\end{document}